\documentclass[a4paper,10pt]{article}
\usepackage{fullpage}
\usepackage[english]{babel}
\usepackage[latin1]{inputenc}
\usepackage{enumerate}
\usepackage{amstext}
\usepackage{epsfig}
\usepackage{amsmath}
\usepackage{amsfonts}
\usepackage{amssymb}
\usepackage{amsthm}
\usepackage{latexsym}
\usepackage{url}
\newcommand{\DD}{\mathbb{D}}
\newcommand{\CC}{\mathbb{C}}
\newcommand{\TT}{\mathbb{T}}

\newcommand{\NN}{\mathbb{N}}
\newcommand{\LL}{\mathcal{L}}
\newcommand{\CP}{\mathbb{CP}}

\newcommand{\parbar}{\overline{\partial}}
\newcommand{\sbars}{\partial S+\overline{\partial S}}
\newtheorem{theorem}{Theorem}[section]
\newtheorem{proposition}{Proposition}[section]
\newtheorem{lemma}{Lemma}[section]
\newtheorem{corollary}{Corollary}[section]
\newtheorem{definition}{Definition}[section]

\begin{document}

\markboth{Carlos Pérez-Garrandés}
{Harmonic currents on certains laminations}

\title{Directed harmonic currents for laminations \\on certain compact complex surfaces}

\author{Carlos Pérez-Garrandés}
\date{April 2013}
\maketitle

\begin{abstract}
Let $\LL$ be a Lipschitz lamination by Riemann surfaces embedded in $M$. If $M$ is a complex torus, $\CP^1\times\CP^1$ or $\TT^1\times\CP^1$ and there is no directed closed current then there exists a unique directed harmonic current of mass one. Moreover if $\LL$ is embedded in $M=\CP^1\times\CP^1$ and has no compact leaves, then there is no directed closed current. If $\LL$ is not Lipschitz, then slightly weaker results are obtained.
\end{abstract}



\begin{section}{Introduction}The aim of this work is to show new examples where the theory of Fornæss and Sibony, developed for homogeneous Kähler manifolds in their fundamental article \cite{FS1}, can be applied. In that article and in \cite{FS2} they face the case of $\CP^2$. In this paper, we consider complex tori and products of curves.\par

Holomorphic non singular foliations in homogeneous Kähler manifolds are studied by Ghys in \cite{Gh1}. There, we can find a classification of foliations in two dimensional tori. Ghys studies the minimal sets of these foliations, and we notice that all of them are holomorphically flat, in the sense of \cite{Oh}. Ohsawa introduces this concept and proves that every Levi-flat is either holomorphically flat or a Levi-scroll, depending on whether the induced foliation on it is minimal or not, respectively. Our main result in this context is: for a transversally Lipschitz lamination embedded on a two dimensional complex torus without directed closed currents there is a unique directed harmonic current of mass one.\par

We will also establish similar results in the case of $\CP^1\times\CP^1$. In fact, we obtain that, for a transversally Lipschitz laminations without compact leaves, there are no directed closed currents, so there is a unique directed harmonic current of mass one. This is related to results in \cite{BG}, where it is proven that a Riccati equation without transversal invariant measure has a unique harmonic measure in the limit set. 
\end{section}
\begin{section} {Revisiting Fornæss-Sibony}
First of all, we will revisit \cite{FS1} to clarify our exposition. Let $(M, \omega)$ be a homogeneous Kähler surface and $T$ will be a real harmonic current of bidegree $(1,1)$ in $M$. Then $T$ can be decomposed as $T=\Omega+\sbars$ where $\Omega$ is a unique closed $\square$-harmonic form of bidegree $(1,1)$, and $S$ is a current of bidegree $(0,1)$ which is not uniquely determined, but $\partial\overline{S}$ is. Moreover, $T$ is closed if and only if $\partial\overline{S}=0$.\par

Since $T=\Omega+\sbars$ with $\Omega$ and $\partial\overline{S}$ uniquely determined, the energy of $T$ is defined as

\begin{equation*}
E(T)=\int{\parbar S\wedge\partial\overline{S}}
\end{equation*}
when $\parbar S$ is in $L^2$. Then $0\leq E(T)<\infty$ and the energy depends only on $T$ but not on the choice of $S$. Considering a scalar product $\langle\, ,\rangle$ on the space of $\square$-harmonic forms, a real inner product and a seminorm  are defined on $\mathcal{H}_e= \{T,$ with $ E(T)<\infty\}$ as 
$$\langle T_1,T_2\rangle_e=\langle\Omega_1,\Omega_2\rangle+\frac{1}{2}\left(\int{\parbar S_1\wedge\partial\overline{S}_2}+{\parbar S_2\wedge\partial\overline{S}_1}\right)$$ 
$$\|T\|_e^2=\langle\Omega,\Omega\rangle+\int{\parbar S\wedge \partial\overline{S}}.$$
With this seminorm we can define a Hilbert space $H_e$ of classes $[T]$as follows: $T_1$, $T_2$ are in the same class if and only if $T_1=T_2+i\partial\parbar u$ with $u\in L^1$ and $u$ real.
\par

Now, for $T_1,T_2$ currents, an intersection form $Q$ is defined by
$$Q(T_1,T_2)=\int{\Omega_1\wedge\Omega_2}-\int{(\parbar S_1\wedge\partial \overline{S_2}+\parbar S_2\wedge\partial \overline{S_1})}.$$
Then $Q(T,T)=\int{\Omega\wedge\Omega}-2E(T)$. This is a continuous bilinear form on $H_e$ and $Q(T,T)$ is upper semicontinuous for the weak topology on $H_e$ and for $T$ harmonic positive current, $Q(T,T)\geq 0$. A class $[T]$ is positive if there is any positive harmonic current in the class $[T]$. Defining the hyperplane $\mathcal{H}=\{T,T\in\, H_e,\int{T\wedge\omega=0}\}$, it can be proven that $Q$ is strictly negative definite on $\mathcal{H}$.\par

Next, this approach is used to study laminar currents. Let $(X,\LL,E)$ be a laminated set with singularities in $(M,\omega)$, a Kähler surface. There exists a unique equivalence class $[T]$ of harmonic currents of mass one directed by the lamination and maximizing $Q(T,T)$ because $Q$ is strictly concave on $\mathcal{H}$. However, this uniqueness is for equivalence classes, not for currents. It is necessary to assume some extra hypotheses:

\begin{theorem}
Let $(X,\LL,E)$ be a laminated set with singularities in a Kähler surface $(M,\omega)$. Suppose $E$ is a locally complete pluripolar set with 2-dimensional Hausdorff measure $\Lambda_2(E)=0$. If there is no non-zero positive laminated closed current, then there is a unique positive harmonic laminated current $T$ of mass one maximizing $Q(T,T)$. 
\end{theorem}

This implies that under the same hypotheses, when $Q(T,T)=0$ for every $T$ positive laminated harmonic current, there exists a unique positive laminated harmonic current of mass one.\par

Finally, the case of a minimal lamination on $\CP^2$ is considered, and it is proven that $Q(T,T)=0$ for every $T$ positive harmonic laminated current when the lamination is Lipschitz or the current has finite transversal energy. That is done defining the geometric intersection between laminated currents, finding some bounds in the number of crossing points between a plaque of the lamination and another plaque moved by an automorphism of $\CP ^2$ such that this geometric autointersection vanishes for every $T$, and proving the vanishing of $Q(T,T)$ by regularizing the current. This reasoning is done for minimal laminations, but the proof is similar if we consider laminations with only one minimal set.\par

Hence, they state

\begin{theorem}
Let $(X,\LL)$ be a $C^1$ laminated compact set in $\CP^2$, without compact curves, then $X$ has a unique positive directed closed harmonic current $T$ of mass 1.
\end{theorem}
The hypothesis mention neither minimality nor closed currents because by Hurder-Mitsumatsu \cite{HM}, absence of compact curves implies no directed positive closed currents, and since Levi problem is solvable on $\CP^2$ and there are no singularities, there is only one $X'\subset X$ minimal set.
\end{section}
\begin{section}{Complex Tori}
We want to study minimal laminations by Riemann surfaces embedded holomorphically in two dimensional tori. Then $\TT^2=\frac{\CC^2}{\Lambda}$, and we have a locally injective projection $\pi:\CC^2\rightarrow \TT^2$ which induces the complex structure on $\TT^2$. Since the embedding is holomorphic, flow boxes are open sets $U$ on $\CC^2$ where $\pi$ is injective and we can write every plaque as a graph of a holomorphic function of $z$ (horizontal flow box) or $w$ (vertical flow box). Explicitly:

\begin{definition} $U\subset\CC^2$ is a horizontal flow box centered on $p=(p_1,p_2)$ where $\pi$ is injective and there is $\delta>0$ and $T_U\subset\CC$ containing $p_2$ such that $\psi_U:\Delta_\delta\times T_U\rightarrow U$ with $\psi(z,w)=(p_1+z,w+f_w(z))$, with $f_w$ a holomorphic function vanishing in $z=0$.
\end{definition} 

\begin{definition} We say a point $p$ of the lamination is horizontal if the unitary tangent vector to the lamination in $p$ is $(1,0)$.\end{definition}

We can define analogously vertical flow boxes and vertical points.\par

\begin{definition}
Let $(X,\LL)$ be a laminated compact set in $\TT^2$. We say that the lamination has invariant complex line segments if there is an affine line $Y$ in $\CC^2$ and $U\subset\CC^2$ open set, such that $Y_{|U}\subset \pi^{-1}(X)_{|U}$.
\end{definition}

Main examples of this kind of laminations are minimal sets for constant vector fields in the torus where every leaf is an affine complex line in the covering $\CC^2$(holomorphically flat laminations). In this sense, there is a paper of Ohsawa \cite{Oh} where he proves that every $C^\infty$ Levi-flat contains a complex segment. Hence if the foliation induced is minimal, then it can only be holomorphically flat.\par

As we know, $\TT^2$ is a complex connected Lie group, so the connected component of the group of automorphisms of $\TT^2$ is $Aut_0(\TT^2)=\TT^2$, and we will denote by $\tau_{(\epsilon_1,\epsilon_2)}(x_1,x_2)=(x_1+\epsilon_1,x_2+\epsilon_2)$ a translation on $\CC^2$ where $x_1,x_2,\epsilon_1,\epsilon_2\in \CC$. These translations induce the automorphisms on $\TT^2$.

\begin{proposition}\label{prop:one}
Let $(X,\LL)$ be a minimal lamination by Riemann surfaces embedded on a torus $\TT^2=\CC^2/ \Lambda$. If there exists $\epsilon_n\rightarrow 0$ such that $\tau_{(0,\epsilon_n)}(\LL)=\LL$, then either every point is vertical or there are no vertical points.
\end{proposition}

\begin{proof}
Suppose there is a $p=(p_1,p_2)$ with vertical tangent. We could find a vertical flow box, $\psi_p:T_p\times\Delta_\delta\rightarrow \CC^2$ where $\psi_p(z,w)\rightarrow(z+f_z(w),p_2+w)$, with $\pi$ injective in $\mathrm{Im}(\psi_p)$. For $n$ big enough $\Gamma_p^{\epsilon_n}$ should be another plaque on the flow box, hence the transversal distance between $d_z(\Gamma_p,\Gamma_p^{\epsilon_n})=|p+f_p(z)-p-f_p(z-\epsilon_n)|>0$ for every $n$ and for every $z$. But this means that $\frac{f_p(z)-f_p(z-\epsilon_n)}{\epsilon_n}$ has no zeros and this sequence converges uniformly to $f'(z)$, which has a zero. Therefore, by Hurwitz's theorem, $f'(z)=0$ for every $z$.
\end{proof}

Hence, if there are no vertical points, the lamination could be covered only by horizontal flow boxes, and for every point $p$ we get a holomorphic function  by analytic continuation, $f_{L_p}$ such that $\pi(z,f_{L_p}(z)),z\in\CC$ parametrizes $L_p$. Likewise, if every point is vertical, the lamination is holomorphically flat.\par

\begin{proposition}
If $L$ is a leaf of a lamination $(X,\LL)$ embedded on a torus $\TT^2$ and there is a holomorphic function $f_L:\CC\rightarrow \CC$ that parametrizes $L$ by $\pi(z,f_L(z))$, then $f_L$ is linear, so $\bar{L}$ is a holomorphically flat laminated set. 
\end{proposition}

\begin{proof}
Firstly, every leaf in $\bar{L}$ is a horizontal graph, because the set of vertical points is a closed set (by Hurwitz's theorem again), and $L$ is dense on $\bar{L}$ without vertical points. So, there is no vertical point in $\bar{L}$. \par

Assuming $f'_L$ is not constant, then there is a sequence $z_n$ with $|f'_L(z_n)|\rightarrow\infty$. But $\pi(z_n,f_L(z_n))$  has a convergent subsequence in $\LL$, $\pi(z_{n_k},f_L(z_{n_k}))\rightarrow (z_0,f(z_0))\in\TT^2$ and the unitary tangent in each point $\pi(z_{n_k},f_L(z_{n_k}))$ is
$\frac{(1,f'(z_{n_k}))}{\sqrt{1+\|f'(z_{n_k})\|}}$. Therefore, this should also converge, and it does, to $(0,1)$, so there would be a vertical point in the lamination. 
This contradiction arises from the fact that $f'_L$ was supposed unbounded. So if it is bounded, by Liouville's theorem, it should be constant. Therefore, $\LL$ is a holomorphically flat lamination.
\end{proof}

Note that all the leaves of holomorphically flat laminations are parabolic, so there is a directed closed current. A lamination on a torus has no invariant complex segments if and only if every leaf  has horizontal and vertical points. Then, in this case, the connected component of the identity of the group of automorphisms of the lamination is  $Aut_0{\LL}=\{id\}$.
\begin{lemma}[Fornæss-Sibony\cite{FS1}]\label{lemma:fs}
\begin{enumerate}
\item There is a constant $0<c_0<1$ such that, for every holomorphic function on the disk $g$ with $|g|<1$, if $g$ has $N$ zeros in $\DD_{1/2}$ then $|g|<c_0^N$ on $\DD_{1/2}$.
\item Let $g$ a holomorphic function on the disk with $|g|<1$. Then, if $|g|<\eta<1$ on $\DD_{1/4}$, then $|g|<\eta^{1/2}$ on $\DD_{1/2}$.
\end{enumerate}
\end{lemma}
We will use this lemma in the proofs of the theorems to estimate the transversal distances.
\begin{theorem}\label{theo:one}
If $(X,\LL)$ is a Lipschitz lamination in $\TT^2$ without invariant complex line segments, we can find a covering by flow boxes $\mathcal{U}$ of $\LL$ and $N\in\NN$ such that, if $\Gamma_\alpha,\Gamma_\beta$ are plaques from the same flow box, then $Card(\Gamma_\alpha\cap\Gamma_\beta^\epsilon)<N$, where $\Gamma_\beta^\epsilon$ is $\Gamma_\beta$ moved by a horizontal or a vertical translation.
\end{theorem}

\begin{proof}
Since $\LL$ has no complex segments, every leaf has vertical and horizontal points. For every vertical point $p_v$, we can take a relatively compact vertical flow box $T_{p_v}\times\Delta_\delta$ such that $f'_\alpha$ has finite number $K_{p_v}$ of zeros on $\Delta_{\delta/2}$ for every $\alpha\in T_{p_v}$ because the lamination has no complex invariant line segments. These flow boxes will be called special vertical flow boxes. Since the set of the vertical points is closed, there is a finite subcovering. We make an analogous argument for horizontal points, and, in this way, collecting all the flow boxes we obtain an open set of the lamination relatively compact, and the complement can be covered by polydisks which can be seen as horizontal or vertical flow boxes for our convenience.\par

Using Hurwitz's theorem like we did in the proof of proposition 3.1, we can find a refinement of this covering, a $k_v\in \NN$ and a $\epsilon_v>0$ small enough such that for every $\alpha,\beta$ in the transversal of a special vertical flow box ($\alpha$ and $\beta$ can be the same point), $f_{\beta}(z-\epsilon)-f_{\alpha}(z)$ has $0<k<k_v$ zeros on $\Delta_\delta$ for every $0<\epsilon<\epsilon_v$. This means that under small vertical translations, plaques in special vertical flow boxes intersect each other at $k$ points.\par

After another refinement, we can assure the same is true for other $\epsilon_h$ and $k_h$ in special horizontal flow boxes. So let be $\epsilon_0=\min(\epsilon_h,\epsilon_v)$ and $k=\max\{k_h,k_v\}$.\par
 
For $\epsilon<\epsilon_0$, $\tau_{(0,\epsilon)}(z,w)=(z,w+\epsilon)$ is a vertical translation, and we suppose we have a horizontal flow box where we have $N'$ intersection points between two plaques, $\Gamma_\alpha,\Gamma_\beta$ when we move one of them by the translation $\Gamma_\beta^\epsilon=\tau_{(0,\epsilon)}(\Gamma_\beta)$. In this case, the transversal distance defined on every $z\in\Delta_\delta$ is
$d_z(\Gamma_\beta,\Gamma_\alpha)=|\alpha+f_\alpha(z)-\beta-f_\beta(z)|$, and as $\LL$ is a Lipschitz lamination, we have $$\frac{|\alpha-\beta|}{C}<d_z(\Gamma_\alpha,\Gamma_\beta)<C|\alpha-\beta|$$ 
for certain global constant $C$ independent of the flow box.
 Since $\Gamma_\alpha$ and $\Gamma_\beta^\epsilon$ intersect, there is $z_0$ with $d_{z_0}(\Gamma_\alpha,\Gamma_\beta)=\epsilon$. Hence $$\frac{\epsilon}{C^2}<d_z(\Gamma_\alpha,\Gamma_\beta)<C^2\epsilon.$$

There is also a constant $b>1$ holding the following: if $\Gamma_1$ and $\Gamma_2$ are two plaques in a flow box with $d_z(\Gamma_1,\Gamma_2)$, the transversal distance on it, and $\Gamma'_1,\Gamma_2'$ are their continuations to an adjacent flow box with the transversal distance $d_z'(\Gamma_1,\Gamma_2)$ then 
$$\frac{\min {d'_z(\Gamma'_1,\Gamma'_2)}}{b}\leq\min {d_z(\Gamma_1,\Gamma_2)}\leq\max{d_z(\Gamma_1,\Gamma_2)}\leq b\max{d'_z(\Gamma'_1,\Gamma'_2)}.$$
This $b$ does not depend on neither the flow box nor the plaques.\par

As we have a finite covering, we can reach a special vertical flow box following a path with at most $M$ changes of flow boxes where $M$ is a global bound. Hence, we get  $$\frac{|\epsilon|}{C^2b^M}<d_z(\Gamma_{\alpha_0},\Gamma_{\beta_0})<C^{2}b^M|\epsilon|$$
where $\alpha_0$ and $\beta_0$ are the analytic continuation of the plaques. \par

Due to the Lipschitzness of the lamination, we can find a global constant $K'$ such that, for every flow box continuating $\Gamma_\alpha$ and $\Gamma_\beta$, say $\Gamma_{\alpha'}, \Gamma_{\beta'}$ we have
$$\frac{d_z(\Gamma_{\alpha'},\Gamma_{\beta'})}{K'|\epsilon|}<\frac{1}{b^2}.$$
By Lemma \ref{lemma:fs}, there is $c<1$ such that
$$\frac{d_z(\Gamma_{\alpha},\Gamma_{\beta}^\epsilon)}{K'|\epsilon|}<c^{N'}<\frac{1}{b^2},$$
then we can see this transversal distance in the next plaques, and considering the distortion, it satisfies that
$$\frac{d'_z(\Gamma_{\alpha'},\Gamma_{\beta'}^\epsilon)}{K'|\epsilon|}<bc^{N'}<1.$$
Hence, in a bigger disk, by Lemma \ref{lemma:fs}, they would differ at most by $(bc^{N'})^{1/2}$. Repeating the argument until we arrive to the vertical special flow box, we get that $d'_z(\Gamma_{\alpha_0},\Gamma_{\beta_0}^\epsilon)<K'|\epsilon|b^2 c^{N'/{2}^M}$.

So, we should have that, by triangular inequality, 
$$d_z(\Gamma_{\beta_0},\Gamma_{\beta_0}^\epsilon)\geq d_z(\Gamma_{\alpha_0},\Gamma_{\beta_0})-d_z(\Gamma_{\alpha_0},\Gamma_{\beta_0}^\epsilon)\geq\left(\frac{1}{C^2b^M}-K'b^2 c^{N'/{2}^M}\right)|\epsilon| $$
but, if $N'$ is big enough to make $\frac{1}{C^2b^M}>K'|\epsilon|b^2 c^{N'/{2}^M}$. It would mean that 
$\Gamma_{\beta'},\Gamma_{\beta'}^\epsilon$ should not intersect each other, but they do. Therefore, making $N=\max\{N',k\}$, we obtain the $N$ appearing in the statement for vertical translations.\par

This argument can be made analogously for horizontal translations.
\end{proof}

\begin{theorem}\label{theo:three}
Let $(X,\LL)$ be a lamination by Riemann surfaces in $\TT^2$ without invariant complex line segments. There exist $\epsilon_0>0$, $A>0$ and a covering $\{B_i\}$ of the lamination, such that for every $\epsilon\in\CC$ with $|\epsilon|<\epsilon_0$, for every $\tau_{(\epsilon,0)}$ and $\tau_{(0,\epsilon)}$,there are at most $A\log{1/|\epsilon|}$ crossing points between $L_1$ and $L_2^\epsilon$ in any flow box.
\end{theorem}

\begin{proof}
The proof is similar to the previous one, but the estimates are slightly different. We will try to be consistent with the notation of the previous theorem. Here, since the lamination is a holomorphic motion, we can take horizontal and vertical flow boxes as we said before, such that 
$$\frac{|\alpha-\beta|^2}{C}\leq |\alpha+f_\alpha(z)-\beta-f_\beta(z)| \leq C|\alpha-\beta|^2.$$
By previous arguments, we can also consider a covering by flow boxes as in the previous theorem, where these inequalities hold for transversal distances, and taking $\epsilon_0$ small enough to assure that a plaque on a special horizontal flow box and the same plaque moved by a horizontal translation have to intersect each other.\par

Proceeding as before, we can check what happens for $\tau_{(\epsilon,0)}$. Assume that we have $N$ crossing points on a vertical flow box. By previous arguments,
$$\frac{\epsilon^4}{K}\leq d_z(\Gamma_\alpha,\Gamma_\beta)\leq K|\epsilon|^{1/4}$$
for certain $K>2$ non depending on $\epsilon$.So, we can reach a special horizontal flow box by a path in at most $M$ changes of flow boxes and $\alpha'$ and $\beta'$ are the corresponding plaques in this flow box. Hence 
$$
\frac{|\epsilon|^{4^M}}{K^Mb^M}<d_z(\Gamma_\alpha',\Gamma_\beta')<b^MK^M|\epsilon|^{1/4^M}.
$$

By similar arguments, we can find a constant $c$ verifiying the estimate $$\frac{d_z(\Gamma_\alpha,\Gamma_\beta^\epsilon)}{K|\epsilon|^{1/4^M}}<c^N<\frac{1}{b^2}.$$
So, by the same reason than for the Lipschitz case,
$$d_z(\Gamma_\alpha',\Gamma_\beta'^\epsilon)<b^2 c^{N/2^M}K|\epsilon|^{1/4^M}.$$
But, by triangular inequality again,
$$d_z(\Gamma_\beta',\Gamma_\beta'^\epsilon)> d_z(\Gamma_\alpha',\Gamma_\beta')-d_z(\Gamma_\alpha',\Gamma_\beta'^\epsilon)>\frac{|\epsilon|^{4^M}}{K^Mb^M}-b^2 c^{N/2^M}K|\epsilon|^{1/4^M}$$
and if 

$$N>\frac{(4^M-(1/4)^M)\log|\epsilon|}{1/2^M\log c}-\frac{\log (2b^{M+2}K^{M+1})}{1/2^M\log c}=A\log{\frac{1}{|\epsilon|}}+B$$

then $d_z(\Gamma_\beta',\Gamma_\beta'^\epsilon)>\frac{|\epsilon|^{4^M}}{2K^Mb^M}>0$, hence $\Gamma_\beta',\Gamma_\beta'^\epsilon$ would not intersect each other. The contradiction arises if $N$ is too big compared to $-\log|\epsilon|$
\end{proof}

This last bound is not as good as the one of Lipschitz case, and we need to recall from \cite{FS1} the following definition.

\begin{definition}A harmonic directed current $T$ which can be written in flow boxes as $T=\int h_\alpha[V_\alpha]d\mu(\alpha)$, has finite transverse energy if in some local flow box $$\int\log|\alpha-\beta|d\mu(\alpha)d\mu(\beta)<-\infty$$
\end{definition}

\begin{theorem}\label{theo:dos} Let $(X,\LL)$ be a lamination in $\TT^2$ without invariant complex line segments. For every harmonic directed current $T$ of mass one with finite transverse energy, $Q(T,T)=0$. This is also true if the lamination is transversally Lipschitz without the assumption of the finiteness of the transversal energy.
\end{theorem}

\begin{proof} This proof follows \cite{FS1} with slight modifications. It is included for the convenience of the reader. We know that if $T$ is a $(1,1)$ positive directed harmonic current it can be seen as $T=\int_A{[V_\alpha]h_\alpha}d\mu(\alpha)$ in a flow box $\Delta\times A$ where $h_\alpha$ is a harmonic function in the plaque $V_\alpha$. Hence, for   $T=\int_A{[V_\alpha]h_\alpha}d\mu(\alpha),T^\epsilon=\int_A{[V^\epsilon_\beta]h'_\beta}d\mu'(\beta)$ where $T^\epsilon=\tau^*_\epsilon(T)$, the geometric intersection is defined in the flow box over a test funtion $\phi$ as 
$$T\wedge_g T^\epsilon(\phi)= \int{\sum_{p\in J_{\alpha,\beta}^\epsilon} h_\alpha(p)h_\beta^\epsilon(p)d\mu(\alpha)d\mu(\beta)}$$
where $J^\epsilon_{\alpha,\beta}$ are the intersection points between $V_\alpha$ and $V_\beta^\epsilon$. Let us suppose that $\mu$ has finite transverse energy. As we proved before, the number of points in $J^{\epsilon}_{\alpha,\beta}$ is at most $A\log1/\epsilon$. As $h_\alpha$ and $h^\epsilon_\beta$ are uniformly bounded, 
$$|T\wedge_g T^\epsilon(\phi)|\leq C_1 \|\phi\|_\infty\int _{d_{min}(V_\alpha,V_\beta)\leq C\epsilon}A\log 1/\epsilon d\mu(\alpha)d\mu(\beta)$$
$$\leq C_2\|\phi\|_\infty\int_{d_{min}(V_\alpha,V_\beta)\leq C\epsilon}\log \frac{1}{d(V_\alpha,V_\beta)} d\mu(\alpha)d\mu(\beta)\rightarrow 0.$$

In the Lipschitz case the number of intersection points is bounded by $N$ independent of $\epsilon$. Therefore, 
$$|(T\wedge_g T_\epsilon)(\phi)|\leq C \|\phi\|_\infty\int_{d_{min}(V_\alpha,V_\beta)\leq C\epsilon} {Nd\mu(\alpha)d\mu(\beta)}\rightarrow 0$$
because $\mu$ has no mass on points.\par

Now, it is necessary to prove that $Q(T,T)=\int T\wedge T=0$. Since we are working on homogeneus Kähler manifolds, it is enough to prove that for smoothings $T^\delta,T^{\delta'}_\epsilon$, $Q(T^\delta,T^{\delta'}_\epsilon)\rightarrow 0$ when $\delta,\delta'$ are small enough compared to $\epsilon$ and $\delta,\delta',\epsilon$ go to $0$.\par

The estimate on the geometric wedge product is stable under
small translations $T_\epsilon$ of $T$. We can think in smoothing a current as an average of small translations.\par

Let $\phi$ be a test function supported in some local flow box. By definition,
the value of the geometric wedge product on $\phi$ is
$$\langle T\wedge T^\epsilon\rangle_g(\phi)= \int{\sum_{p\in J_{\alpha,\beta}^\epsilon} h_\alpha(p)h_\beta^\epsilon(p)d\mu(\alpha)d\mu(\beta)}.$$
But if we fix a plaque $[V_\beta^\epsilon]$ we can look for points in it which are also points of a plaque $V_\alpha$ and we write it like this
$$\langle T\wedge T_\epsilon\rangle_g(\phi)=\int\left(\int_{V_\beta^\epsilon}[\phi h_\alpha h_\beta^\epsilon](p)i\partial\parbar\log|w-f_\alpha(z)|d\mu(\alpha)\right)d\mu(\beta)$$
These expressions are small when $\epsilon$ is small. The same applies when we do this for slight translations within small neighbourhoods $U(\epsilon)$ of the identity in $Aut_0(\TT^2)=\TT^2$ and their smooth averages $T^\delta$. So, if we consider $\phi T^\delta$ as a smooth test form we get
$$\langle T_\epsilon,\phi  T^\delta\rangle=\int\left(\int_{V_\beta^\epsilon} [\phi h^\epsilon_\beta](p)T^\delta\right)d\mu(\beta).$$
Repeating the process, consider the averaging over small translations of $T_\epsilon$, we get that $T^{\delta'}_\epsilon\wedge T^\delta(\phi)\rightarrow 0$ when $\delta,\delta'<<\epsilon$ and $\epsilon\rightarrow 0$. Since this argument is made over flow boxes, we need to consider a partition of unity of $\phi$'s, so we get $T^{\delta'}_\epsilon\wedge T^\delta=Q(T^{\delta'}_\epsilon,T^\delta)\rightarrow 0$. Therefore $Q(T,T)=0$.
\end{proof}

Non existence of directed positive closed currents implies $(X,\LL)$ has no invariant complex line segments. Hence $Q(T,T)=0$ for every $T$ harmonic directed positive current. This observation leads us to the following corollary.

\begin{corollary}
 Let $(X,\LL)$ be a Lipschitz lamination in $\TT^2$ with no directed positive closed currents. Then there is a unique harmonic current $T$ of mass one directed by the lamination. Specifically it has only one minimal set.
\end{corollary}
\end{section}
\begin{section}{Products of curves}
In this section we will deal with cases of some products of curves with many automorphisms. They are $\CP^1\times \CP^1$ and $\CP^1\times\TT^1$. We have a slightly different definition of verticality and horizontality here, but it is natural anyway based on their standard parametrizations. We define $\phi_1:\CC\rightarrow \CP^1$ as $\phi_1(w)=[1:w]$ and $\phi_2:\CC\rightarrow \CP^1$ as $\phi_2(z)=[z:1]$. For $\TT^1$, since $\pi:\CC\rightarrow \TT^1$ is locally injective, there exists $\delta>0$ such that $\pi_{|\Delta_\delta(z)}$ is injective for every $z\in\CC$. So, every $p$ in $X=\CP^1\times\CP^1, \TT^1\times\CP^1$ admits a parametrization $\varphi=(\varphi_1,\varphi_2)$ where $\varphi_i$ are injective restrictions to disks of these functions.

\begin{definition} $U\subset X$ is a horizontal flow box centered on $p=(p_1,p_2)$ if there is a parametrization as we said before with $\varphi(z_0,w_0)=(p_1,p_2)$, a disk $D_1$ centered at $0$, a subset $A$ contained on a disk $D_2$ centered at $0$ such that plaques of $\LL_{|U}$ are parametrized by $\varphi(z_0+z,w_0+\alpha+f_\alpha(z))$ for every $\alpha\in A$.
\end{definition} 

\begin{definition} 
We will say that a point $p$ of the lamination is horizontal if $\pi_2(T_p\LL)=0$.
\end{definition}

We define analogously vertical flow boxes and vertical points, and we can cover our lamination by horizontal or vertical flow boxes. Note that if $p$ is a horizontal point we can take a horizontal flow box on a neighbourhood of $p$, and if $\varphi(z_0,w_0)=p$, then $f_0'(0)=0$. \par

\begin{proposition} Every minimal lamination $(X,\LL)$ in $\TT^1\times\CP^1$ either has horizontal points, or is $\TT^1\times\{p\}$. And if $(X,\LL)$ is in $\CP^1\times\CP^1$ and there is a leaf $L$ without horizontal points, then $L=(f(p),p)$ is a closed leaf for $f:\CP^1\rightarrow\CP^1$ holomorphic.
\end{proposition}

\begin{proof}
The proof is analogous to Proposition \ref{prop:one}. We can consider a covering only with vertical flow boxes and, beginning with a vertical plaque $\Gamma_\alpha$ with a parametrization $\varphi(f_\alpha(z),z)$, we can extend $f_\alpha$ to obtain a holomorphic function from $\CP^1$ to the first factor of the surface. If the first factor is $\CP^1$, this function is rational, but if the first factor is $\TT^1$ there are no nonconstant holomorphic functions from $\CP^1$ to $\TT^1$.\par
\end{proof}

Clearly, the same is true for vertical points in $\CP^1\times\CP^1$. So every lamination $(X,\LL)$ embedded in it without compact curves has vertical and horizontal points.\par

\begin{theorem}
Let $(X,\LL)$ be a lamination without compact leaves having only one minimal set in $M=\CP^1\times\CP^1$ . Suppose that the point $p=([0:1],[1:0])$ is neither vertical nor horizontal and belongs to the minimal set. Let $\Phi_\epsilon$ be the automorphism of $\CP^1\times\CP^1$ defined as $\Phi_\epsilon([z_1:z_2],[w_1:w_2])=([z_1+\epsilon z_2:z_2],[w_1:w_2])$. Then, there exists a covering and some constants $\epsilon_0>0$, $N\in\NN$, $A>0$ such that
if $\Gamma_\alpha,\Gamma_\beta$ are plaques from the same flow box and

\begin{enumerate}
\item $(X,\LL)$ is a Lipschitz lamination then $Card(\Gamma_\alpha\cap\Gamma_\beta^\epsilon)<N$ with $\Gamma_\beta^\epsilon=\Phi_\epsilon(\Gamma_\beta)$
\item $(X,\LL)$ is not Lipschitz then $Card(\Gamma_\alpha\cap\Gamma_\beta^\epsilon)<A\log1/\epsilon$ 
\end{enumerate}

for every $\epsilon$ with $|\epsilon|<\epsilon_0$
\end{theorem}

\begin{proof}
We will explain the Lipschitz case. The only one difference with non Lipschitz case is that the last one has a little bit more complicated inequalities as we could see in Theorem \ref{theo:three}.\par

First of all, we notice that $[1:0]\times \CP^1$ is invariant for every $\Phi_\epsilon$ and we consider a flow box centered on $p$. 
\begin{align}
\varphi:\Delta_\delta\times A&\rightarrow B_0\subset\CP^1\times\CP^1\\
(z,w)&\rightarrow([1:z],[f_w(z)+w:1])
\end{align}

small enough to hold that $0<\frac{f_0'(0)}{2}<f_w'(z)<2f'_0(0)$. We cover $\CP^1\times\CP^1\setminus{B_0}$ by horizontal or vertical flow boxes, and we obtain a covering $\mathcal{B}=\{B_i\}$. \par

The automorphism $\Phi_\epsilon$ sends $(z,w)$ to $(\frac{1}{1+\epsilon z},w)$, so the transversal distance between a plaque  $\Gamma_\beta$ of $\LL$ and $\Gamma_\beta^\epsilon$ the same one moved by $\Phi_\epsilon$,  is  

\begin{align*}
d_z(\Gamma_\beta,\Gamma_\beta^\epsilon)&=\left|\beta+f_\beta(z)-\beta-f_\beta\left(\frac{1}{1+\epsilon z}\right)\right|&\\
&=\left|f_\beta(z)-f_\beta\left(\frac{1}{1+\epsilon z}\right)\right|&\\
& \geq k \left|z-\frac{1}{1+\epsilon z}\right|\\
&=k \left|\epsilon\frac{z^2}{1+\epsilon z}\right|
\end{align*}

Therefore, $\max_{|z|\leq\delta}d_z(\Gamma_\beta,\Gamma_\beta^\epsilon)= \max_{|z|=\delta}d_z(\Gamma_\beta,\Gamma_\beta^\epsilon)\geq k|\epsilon||\delta|^2/2$ if $\epsilon$ small enough.\par

Now, we repeat the argument. Consider two plaques $\Gamma_\alpha$ and $\Gamma_\beta^\epsilon$ which intersect each other in $N$ points. Following a path, we reach $B_0$ in at most $M$ changes of flow boxes which is independent of the plaques. Let $\alpha'$ and $\beta'$ be the analytic continuation of the original plaques, and by same reasoning of Theorem \ref{theo:one}, we obtain that, $d_z(\Gamma_{\beta'}^\epsilon,\Gamma_{\alpha'})\leq K'|\epsilon|b^2 c^{N'/{2}^M}$ for every $z\leq\delta$, in fact for $z=0$, $d_0(\Gamma_{\beta'}^\epsilon,\Gamma_{\alpha'})=|\alpha'-\beta'|\leq K'|\epsilon|b^2 c^{N'/{2}^M}$, so $$d_z(\Gamma_{\alpha'},\Gamma_{\beta'})\leq C|\alpha'-\beta'|\leq CK'|\epsilon|b^2 c^{N'/{2}^M}.$$ Finally, 
\begin{align*}
\left|\frac{k\epsilon\delta^2}{2}\right|&\leq\max_{|z|\leq z}d_z(\Gamma_{\beta'}^\epsilon,\Gamma_{\beta'})\\
&\leq\max_{|z|\leq z}d_z(\Gamma_{\beta'}^\epsilon,\Gamma_{\alpha'})+\max_{|z|\leq z}d_z(\Gamma_{\beta'},\Gamma_{\alpha'})\\
&\leq K'|\epsilon|b^2 c^{N'/{2}^M}+C K'|\epsilon|b^2 c^{N'/{2}^M}
\end{align*} 

then if $N$ is big enough to hold $k|\epsilon||\delta|^2/2>(C+1)K'|\epsilon|b^2 c^{N'/{2}^M}$, a contradiction arises. So the number of intersection points is bounded
\end{proof}

\begin{theorem}
Let $(X,\LL)$ be a lamination without compact leaves and having only one minimal set embedded in $M=\TT^1\times\CP^1$. Let $\Phi_\epsilon([z_1],[w_1:w_2])=([z_1+\epsilon],[w_1:w_2])$ then there exists a covering of $\LL$ by flow boxes and some constants $N\in\NN$, $\epsilon_0>0$ and $A>0$ such that
if $\Gamma_\alpha,\Gamma_\beta$ are plaques from the same flow box and  $\Gamma_\beta^\epsilon=\Phi_\epsilon(\Gamma_\beta)$, 
\begin{enumerate}
\item $(X,\LL)$ is a Lipschitz lamination then $Card(\Gamma_\alpha\cap\Gamma_\beta^\epsilon)<N$ 
\item $(X,\LL)$ is not Lipschitz then $Card(\Gamma_\alpha\cap\Gamma_\beta^\epsilon)<A\log1/\epsilon$ 
\end{enumerate}

for $|\epsilon|<\epsilon_0$.
\end{theorem}

\begin{proof}
The proof of this theorem is similar to theorems \ref{theo:one} and \ref{theo:three}. Since $\LL$ has no compact leaves, there are non horizontal points. Hence, we just need to take a finite covering of the horizontal points by special horizontal flow boxes, find a $\epsilon_0$ small enough to hold that every plaque in these flow boxes intersects itself when we move it by $\Phi_\epsilon$ if $|\epsilon|<\epsilon_0$, and get the same contradiction we obtain in theorems \ref{theo:one} and \ref{theo:three}.
\end{proof}

\begin{theorem}
 Let $(X,\LL)$ be a lamination without compact leaves having only one minimal set in $\CP^1\times\CP^1$ or $\TT^1\times\CP^1$. For every harmonic directed current $T$ of mass one with finite transverse energy, $Q(T,T)=0$. This is also true if the lamination is Lipschitz without the assumption of the finiteness of the transversal energy.
 \end{theorem}
 
The proof is analogous to \cite{FS1} and theorem \ref{theo:dos}.

\begin{corollary}
Let $(X,\LL)$ be a lamination without compact leaves having only one minimal set in $M=\CP^1\times\CP^1$ . Then there is no closed laminar current of mass one if $\LL$ is Lipschitz and, if the lamination is not Lipschitz, every closed laminated current of mass one has infinite transverse energy.
\end{corollary}

\begin{proof}
We know that if $T$ is closed of mass one $T=\Omega+\partial S+\overline{\partial S}$ for a unique $\square-$harmonic form $\Omega$ and $\partial \overline{S}=0$. So $\int{T\wedge T}=\int{\Omega\wedge\Omega}=0$, but we are on a Kähler manifold where the dimension of $H^{1,1}(M)$ is two. It is generated by the Kähler form $\omega$ and another form $\omega'$ such that $\int{\omega\wedge\omega}=1, \int{\omega'\wedge\omega'}=-1$ and $\int{\omega'\wedge\omega}=0$. So, every positive laminar closed current $T$ of mass one has $\Omega=\omega+\omega'$.\par

Now, consider the automorphism of $M$, $S(p_0,p_1)=(p_1,p_0)$. $S_*(T)$ is a laminar current for the lamination $S(\LL)$ which has similar properties to $\LL$.\par $Q(T,S(T))=\int{\Omega\wedge\Omega}=0$ following \cite{FS1} but, by \cite{Du}, the product of closed laminar currents is always geometric. So, we consider a flow box together for $\LL$ and $S(\LL)$ on a neighbourhood of 
$([0:1],[0:1])$ where $\LL$ is horizontal, so $S(\LL)$ is vertical and $T=\int_{\alpha\in A} {[V_\alpha] d\mu(\alpha)}$ where $V_\alpha$ are the plaques and $S(T)=\int_{\alpha\in A} {S^*[V_\beta] d\mu(\beta)}$ with the same $\mu$ then $T\wedge_g S(T)(\Phi)=\int_{A\times A} \Phi d\mu d\mu$ for a test form $\Phi$ but, as the intersection is geometric, this is $0$ for every test form, hence $T=0$. 
\end{proof}

We can also assure that when we have no closed directed currents, there is a unique harmonic positive current directed by the lamination of mass one. This remark in $\CP^1\times\CP^1$ is coherent with \cite{BG}, where it is proven that there is a unique harmonic foliated measure for Riccati foliations when there are no holonomy invariant measures.

\end{section}
\begin{section}*{Acknowledgments}
Partially supported by Ministerio de Ciencia e Innovación (Spain), MTM2010-15481. The author is supported by a predoctoral grant of Universidad Complutense de Madrid.\par

This paper will be part of my Ph.D. thesis directed by Luis Giraldo and John Erik Fornæss. I thank Professor Giraldo for his advice, encoraugement, support and careful reading of this article and to Professor Fornæss for his help and for so many suggestions, comments and explanations which have been essential in the developement of these results. I also want to thank to the referee, whose suggestions were helpful in order to improve the generality of the theorems as well as the clarity of the arguments.
\end{section}
\bibliographystyle{plain}
\bibliography{mybib}{}
\vspace{1cm}
\begin{flushleft}
C. Pérez-Garrandés \\
Instituto de Matemáticas Interdisciplinar (IMI)\\
Departamento de Geometría y Topología\\
Facultad de Ciencias Matemáticas\\
Universidad Complutense de Madrid\\
              Plaza de las Ciencias 3\\
              28040 Madrid\\
           carperez@ucm.es\\
          \end{flushleft}

\end{document}